\documentclass[12pt,twoside]{article}
\usepackage{graphicx,graphics, amsfonts, amsbsy,amssymb, amsmath, cite, array,arydshln, hyperref, float}
\allowdisplaybreaks

\newtheorem{theorem}{Theorem}[section]

\newtheorem{lemma}[theorem]{Lemma}
\newtheorem{corollary}[theorem]{Corollary}

\newtheorem{proposition}[theorem]{Proposition}

\newtheorem{conjecture}{Conjecture}

\usepackage{tikz}
\usetikzlibrary{chains,positioning}
\usepackage{amsmath,tikz,pgfplots}
\usetikzlibrary{decorations.pathmorphing}

\def\qed{\nolinebreak\hfill\rule{.2cm}{.2cm}\par\addvspace{.5cm}}

\begin{document}
\title{On (distance) Laplacian characteristic polynomials of power graphs}
\author{Bilal Ahmad Rather$ ^{a} $, Mustapha Aouchiche$ ^{b},$ Victor A. Bovdi$^{a,}$\footnote{Corresponding Author} \\
	$ ^{a} $\emph{Mathematical Sciences Department, COS, UAE University, UAE}\\
	\texttt{bilalahmadrr@gmail.com, vbovdi@gmail.com}\\ $^{b}$\emph{GERAD and Polytechnique Montreal, Montreal, QC, Canada} \\
	\texttt{maouchiche402@gmail.com}
	}
\date{}

\pagestyle{myheadings} \markboth{}{On (distance) Laplacian characteristic polynomial of power graphs}
\maketitle
\vskip 5mm
\noindent{\footnotesize \bf Abstract.} The characteristic polynomials of the Laplacian  and the distance Laplacian matrices of power graphs of groups of order $ pqr  $, where $ p,q $ and $ r $ are { primes,} are obtained. Further, the characteristic polynomials of these matrices for proper power graphs of cyclic and dicyclic groups are given. The important inequalities for the zeros of the distance Laplacian characteristic polynomials of power graphs of finite groups are presented in  comments.
\vskip 3mm

\noindent{\footnotesize Keywords: Laplacian matrix, distance Laplacian matrix, characteristic polynomial, power graph}

\vskip 3mm
\noindent {\footnotesize AMS subject classification: 05C50, 05C12, 05C25, 15A18.}

\section{Introduction}
\paragraph{}
Let $ G $ be a simple undirected finite connected graph with  vertex set $V(G)=\{v_{1},\ldots,v_{n}\}$ and edge set $E(G)$ that consists of unordered pairs of vertices. The cardinality of $V(G)$ is the \textit{order} $ n $ and the cardinality of  $E(G)$ is the \textit{size} of $G$. For $v\in V(G)$, the \textit{degree} $d_{G}(v)$ of $v$ in $ G $ is the number of edges incident on $ v. $ The union $ G_{1}\cup G_{2} $ of graphs $ G_{1} $ and $ G_{2} $ is the graph with vertex set $ V(G_{1})\cup V(G_{2}) $ and edge set $ E(G_{1})\cup E(G_{2}) $. A join of two $ G_{1} $ and $ G_{2} $, denoted by $ G_{1}\vee G_{1} $ is obtained by taking $ G_{1}\cup G_{2} $ and adding all edges $ \{u,v\}, $ where $ u\in V(G_{1}) $ and $ v\in V(G_{2}). $ By $ K_{n} $ and $ \overline{K}_{n} $, we denote the complete graph and its complement, respectively.

The adjacency matrix $A=(a_{ij})$ of $G$ is a $(0, 1)$-square matrix of order $n$ whose $(i,j)$-entry is equal to 1, if $v_i$ is adjacent to $v_j$ and equal to 0, otherwise. Let $\text{Deg}(G)={\text{diag}}(d_1, \dots, d_n)$ be the diagonal matrix of vertex degrees $d_i=d_{G}(v_i)$, for $i=1,\dots,n.$ The matrix $L(G)=\text{Deg}(G)-A(G)$ is called the Laplacian matrix and the polynomial $ \Theta_{L}(G,x)=\det(L(G)-xI_{n}) $ is the Laplacian characteristic polynomial of  $G$, where $ \det $ stands for determinant and $ I_{n} $ is the identity matrix. The polynomial $ \Theta_{L}(G,x) $ has only real zeros $ \mu_{1}\geq \mu_{2}\geq \dots \geq \mu_{n-1}\geq \mu_{n}=0, $ which are known as the Laplacian eigenvalues of $ G $.  
If $ \mu_{n-1}>0 $, then  the zero $\mu_{n-1}$ is called { the \emph{algebraic connectivity} of $G$}. Note that $\mu_{n-1}>0$ if and only if $ G $ is connected \cite{fiedler}. 
More literature about these matrices can be found in the book \cite{cds}.

In a connected graph $G$, the \textit{distance} $d(u,v)$ between two vertices $u\neq v\in V(G),$ is the length of a shortest path between $u$ and $v$. The \textit{diameter} of $G$ is the maximum distance between any two vertices of $G.$ The \textit{distance matrix} $D(G)$ of $G$ is defined as $D(G)=(d(u,v))_{u,v\in V(G)}$. For more about $D(G)$, see the survey papers \cite{AH1,huiqiu}. The \textit{transmission} $Tr_{G}(v)=\sum\limits_{u\in V(G)}d (u,v)$ of a vertex $v$ is the sum of the distances from $v$ to all other vertices in $G$. 
For any vertex $v_i\in V(G)$, the transmission $Tr_G(v_i)$ is called the \textit{transmission degree}, shortly denoted by $Tr_{i}$ and the sequence $\{Tr_{1},,\ldots,Tr_{n}\}$ is called the \textit{transmission degree sequence} of  $G$.

Let $Tr(G)=\text{diag} (Tr_1,\ldots,Tr_n) $ be a diagonal matrix of vertex transmissions of $G$. The matrix $  D^L(G)=Tr(G)-D(G)  $ is called the distance Laplacian  matrix  of $ G $ \cite{AH2}. The polynomial $ \Theta_{D^{L}}(G,x)=\det(D^L(G)-xI_{n}) $ is called the distance Laplacian characteristic polynomial of $ G. $ The distance Laplacian eigenvalues of $ G $ are  the zeros of $ \Theta_{D^{L}}(G,x) $ and are indexed  as: $ \partial_{1}^{L}\geq\partial_{2}^{L}\geq \dots\geq\partial_{n-1}^{L}\geq\partial_{n}^{L}=0. $ More about the distance Laplacian matrix can be found in \cite{AH2, AH3, AH4, bilalcmj}.

Kelarev and Quinn \cite{kelarev} introduced the notion of power graph of a semigroup $ \mathcal{S} $ as a directed graph with vertex set $ \mathcal{S} $ and there is an arc from a vertex $ x $ to other vertex $ y $ if $ y=x^{i} $ for some positive integer $ i. $ Later, Chakrabarty et al. \cite{sen} defined the undirected power graph $ \mathcal{P}(\mathcal{G}) $ of a group $\mathcal{G} $ as an undirected graph with vertex set $ \mathcal{G} $ and two vertices $ x\neq y\in \mathcal{G} $ are adjacent if and only if one is the power of other, that is, $ x^{i}=y $ or $ y^{j}=x $, for some positive integers $i $ and  $ j $. The identity of the group $ \mathcal{G} $ is denoted by $ e $. The additive group (cyclic) of integer modulo $ n $ is denoted by $ \mathbb{Z}_{n} $. We denote $ \mathcal{G}^{*}=\mathcal{G}-e $ and the proper power graph of $ \mathcal{P}(\mathcal{G}) $ is $ \mathcal{P}(\mathcal{G}^{*})=\mathcal{P}(\mathcal{G}-e) $, that is, power graph obtained by removing the vertex $ e $. The first survey \cite{survey} on power graphs was done in (2013) and the more recent one in (2021) by Kumar et. al. \cite{survey1}. Adjacency spectrum of power graphs of various finite groups were given in \cite{mehreen1}. Laplacian spectrum of power graphs of finite cyclic and dihedral groups has been investigated in \cite{sriparna,panda}. More about the spectral properties of graph matrices of power graphs can be see in \cite{asma, mehreen1, panda,  bilalCMP, bilalcjm, bilaljaa}.

In general, it is not always possible to find the zeros of the Laplacian or the distance Laplacian characteristic polynomials of a graph.  However many authors have considered this problem for some special classes of graphs.  Likewise, it is very difficult to find the Laplacian or the distance Laplacian characteristic polynomial of power graphs of finite groups, though many researchers have tried it for some special classes of power graphs, see \cite{bilalacta, mehreen1,sriparna, ghorbani1, bilaldmaa}.  Here in this article, we continue investigating this problem for the power graphs of the finite groups of order $ pqr,  $ where $ p\geq q \geq r $ are primes and for proper power graphs of certain finite groups.

In Section \ref{section 2}, we discuss the Laplacian and the distance Laplacian characteristic polynomials of the power graph of the groups of order $ pqr $ $ (p\geq q\geq r) $. In Section \ref{section 3}, the characteristic polynomials of the Laplacian and the distance Laplacian matrices of the proper power graphs of the cyclic group $ \mathbb{Z}_{n} $ and the dicyclic group are given. Finally, we close the article with some comments, which gives important eigenvalue inequalities for the Laplacian and the distance Laplacian matrices of power graphs.
\section{Distance Laplacian characteristic polynomials of power graphs}\label{section 2}
\paragraph{}
Consider an $ l \times l $ matrix
\[M= \left( \begin{array}{c | c | c | c | c}
B_{1,1} & B_{1,2} & \cdots & B_{1,k-1} & B_{1,k} \\
\hline
B_{2,1} & B_{2,2} & \cdots & B_{2,k-1} & B_{2,k} \\
\hline
\vdots  & \vdots  & \ddots & \vdots & \vdots \\
\hline
B_{k-1,1} & B_{k-1,2} & \cdots & B_{k-1,k-1} & B_{k-1,k} \\
\hline
B_{k,1} & B_{k,2} & \cdots & B_{k,k-1} & B_{k,k} \\
\end{array} \right),
\]
whose rows and columns are partitioned according to a partition $P=\{ P_{1}, \ldots , P_{k}\} $ of $X= \{1,\ldots,l\}. $ The \emph{quotient matrix} $ Q=(q_{ij})_{k\times k} $ \cite{BH} of $ M $ is a matrix, whose $ ij $-th entries are the average row sums of the blocks $ B_{i,j} $ of $ M. $ The partition $ P $ is called  \textit{equitable} if each block $ B_{i,j} $ of $ M $ has constant row (column) sum and in such case, the matrix $ Q $ is known as the \textit{equitable quotient matrix}. In general, the eigenvalues of $ Q $ interlace the eigenvalues of $ M $, while for equitable partition, each eigenvalue of $ Q $ is the eigenvalue of $ M. $

Let $ G $ be a connected graph. A real valued vector $ X=(x_{1}, \dots,x_{n})^{T} \in \mathbb{R}^{n} $ can be regarded as a function defined on $ V(G) $ which maps $ v_{i} \mapsto  x_{i} $  for each $ i=1,\ldots,n. $ Evidently
 \[ X^{T}D^{L}(G)X=\sum\limits_{\{u,v\}\subseteq V(G)}d(u,v)\big(x_{u}-x_{v}\big)^{2}.\]
The number $ \partial $ is an eigenvalue of  $ D^{L} $ matrix with the corresponding eigenvector $X $ if and only if
\[
\partial x_{v}=\sum\limits_{u\in V(G)}d(u,v)\big(x_{v}-x_{u}\big), \qquad \qquad (v\in V(G))
\]
or equivalently,  \begin{equation}\label{eigen equation}
\partial x_{v}=Tr(v)x_{v}-\sum\limits_{u\in V(G)}d(u,v)x_{u}, \qquad \qquad (v\in V(G)).
\end{equation}
Equation \eqref{eigen equation} is known as the $ (\partial, X) $-eigenequation of $ D^{L}(G) $ at $ v. $

The next  results from \cite{AH3}  are helpful in finding some zeros of $\Theta_{D^{L}}(G,x) $. 
\begin{lemma}[\cite{AH3}]\label{eigenvalue of independence number} Let $ G $ be a graph of order $ n $. If $ S=\{v_{1}, \ldots, v_{s}\} $ is an independent set of $ G $ satisfying $ N(v_{i})=N(v_{j}) $ for every $ i,j\in \{1,\ldots,s\} $. Then $ \partial=Tr(v_{i})=Tr(v_{j}) $ for each $ i,j\in \{1,\ldots,s\} $ and $ (x-\partial-2) $ is a factor of $\Theta_{D^{L}}(G,x) $ with multiplicity at least $ s-1. $
\end{lemma}

\begin{lemma}[\cite{AH3}]\label{eigenvalue of clique} Let $ G $ be a graph of order $ n $. If $ \mathcal{W}=\{v_{1}, \dots, v_{\omega}\} $ is a clique of $ G $ satisfying $ N(v_{i})-\mathcal{W}=N(v_{j})-\mathcal{W} $ for every $ i,j\in \{1,\dots,\omega\} $, then $ \partial=Tr(v_{i})=Tr(v_{j}) $ for each $ i,j\in \{1, \dots,\omega\} $ and $ (x-\partial-1) $ is a factor of  $ \Theta_{D^{L}}(G,x) $  with multiplicity at least $ \omega-1. $

\end{lemma}

The next result, relates the characteristic polynomials of the Laplacian matrix to the distance Laplacian matrix for graphs of diameter less or equal to two.
\begin{lemma}[\cite{AH2}]\label{diameter 2 lemma}
Let $ G $ be a connected graph of order $ n $  with diameter at most two and let
\[ x(x-\lambda_{1})(x-\lambda_{2})\cdots(x-\lambda_{n-2})(x-\lambda_{n-1}) \] be its Laplacian characteristic polynomial. Then the distance Laplacian characteristic polynomial of $ G $ is
\[ x(x-2n+\lambda_{n-1})(x-2n+\lambda_{2n-2})\cdots(x-2n+\lambda_{2})(x-2n+\lambda_{1}). \]
\end{lemma}

Let $ M $ be a matrix associated to a graph $ G $. We say that $ G $ is $ M $-integral if the spectrum of $ M $ is integral, that is, the zeros of the characteristic polynomial $ \det(M-xI_{n}) $ are all integers. The following proposition is an immediate consequence of Lemma \ref{diameter 2 lemma}.
\begin{proposition}\label{Laplacian integrability}
Let $ G $ be a connected graph of order $ n $ with diameter at most $ 2. $ Then $ G $ has the integral Laplacian zeros if and only if $ G $ has the distance Laplacian integral zeros.
\end{proposition}

Next, we have a useful lemma that will be useful for our main results.
\begin{lemma}\label{lemma equal componenet}
Let $ G $ be a connected graph with non-zero eigenvector  $ X=(x_{1},\dots,x_{n})^{T} $ corresponding to the  distance Laplacain eigenvalue $ \partial $. If $ u,v\in V(G) $ is such that $ N(u)\setminus\{v\}=N(v)\setminus \{u\} $, then $ x_{u}=x_{v}. $
\end{lemma}
\textbf{Proof.} For $ u,v\in V(G) $ with $ N(u)\setminus\{v\}=N(v)\setminus \{u\} $, it follows that  $ Tr(u)=Tr(v) $, so by Equation \eqref{eigen equation}, we have
\[ \partial x_{u}-Tr(u)x_{u}=-d(u,v)x_{v}-\sum\limits_{v_{j}\in V(G), v_{j}\neq v,u }d(u,v_{j})x_{j}, \]
and
\[ \partial x_{v}-Tr(v)x_{v}=-d(v,u)x_{u}-\sum\limits_{v_{j}\in V(G), v_{j}\neq v,u }d(v,v_{j})x_{j}. \]
As $ N(u)\setminus\{v\}=N(v)\setminus \{u\} $, so $ d(u,v_{j})=d(v,v_{j}) $ for $ v_{j}\neq u,v. $ Therefore, it follows that
\[ \partial(x_{u}-x_{v})=d(u,v)(x_{u}-x_{v}), \]
which implies that $ x_{u}=x_{v}. $ \qed

Let $ G_{i}(V_{i}, E_{i})$ be graphs of order $m_i,$ where $i=1,\ldots, n $. The  \textit{joined union} \cite{DS} is the graph $ H(W, F) $ with
\[W=\bigcup_{i=1}^{n}V_{i} \qquad
\text{and}\qquad F=\bigcup_{i=1}^{n}E_{i}\cup\bigcup_{\{v_{i}, v_{j}\}\in E}V_{i}\times V_{j}.
\]
In other words, the joined union is the union of graphs $ G_{1},\ldots, G_{n} $ together with the edges from every vertex  in $ G_{i} $  to every vertex  in $ G_{j} $, whenever $ v_{i} $ and $ v_{j} $ are adjacent in $ G. $ For example, the join of two graphs $ G_1 $ and $G_2$ is $G_{1}\vee G_{2}=K_{2}[G_1, G_2]. $

Denote by $ G(p,q,r) $, the class of all groups of order $ pqr $, where $ p,q $ and $r$ are primes. H\"{o}lder \cite{holder} (see, also \cite{ghorbani1}) investigated the structures of groups in $ G(p,q,r) $. For $ p=q=r $, there are the following five non-isomorphic  groups  of order $ p^{3} $:
\[
\mathbb{Z}_{p^{3}}, \quad  \mathbb{Z}_{p}\times \mathbb{Z}_{p^{2}}, \quad \mathbb{Z}_{p}\times\mathbb{Z}_{p}\times \mathbb{Z}_{p}, \quad  \mathbb{Z}_{p}\rtimes\mathbb{Z}_{p^{2}}, \quad  \mathbb{Z}_{p}\rtimes(\mathbb{Z}_{p}\times \mathbb{Z}_{p}). \]
where $ \times $ is the direct product of groups and $ \rtimes $ is the semi direct product of groups.
For $ p>q>r $, the structure of non-isomorphic groups of order $ pqr $ are
\begin{align*}
	&  \mathbb{Z}_{pqr},  \mathbb{Z}_{r}\times F_{p,q}, (p\equiv 1\pmod{q}),\qquad  \mathbb{Z}_{q}\times F_{p,r}, (p\equiv  1\pmod {r}),\\
& \mathbb{Z}_{p}\times F_{q,r}, (q\equiv 1\pmod {r}),\qquad  F_{p,qr}, (p\equiv 1 \pmod {qr}),
\end{align*}
and $G_{i+5}\cong\langle a,b,c : a^{p}=b^{q}=c^{r}=1, ab=bc, c^{-1}bc=b^{u}, c^{-1}ac=a^{v^{i}} \rangle$, where $ p\equiv 1\pmod {r}$ , $q\equiv 1 \pmod {r}$, $o(u)=r \in \mathbb{Z}_{q}^{*} $ and $ o(v)=r \in \mathbb{Z}_{p}^{*}$ for  $1\leq i\leq r-1$.

Based on the description of these groups of order $ pqr $, we find the Laplacian and the distance Laplacian characteristic polynomials of their power graphs.

Our first result gives the distance Laplacian characteristic polynomial of $ \mathcal{P}(\mathbb{Z}_{p}\times \mathbb{Z}_{p^{2}}). $
\begin{theorem} \label{theorem p x p^2}
The distance Laplacian characteristic polynomial of the power graph $ G\cong \mathcal{P}(\mathbb{Z}_{p}\times \mathbb{Z}_{p^{2}}) $ of the group $ \mathbb{Z}_{p}\times \mathbb{Z}_{p^{2}} $ is
\begin{align*}
\Theta(G,x)=&x(x-p^{3})(x-2p^{3}+p)^{p^{2}-p-1}(x-p^{3}-p^{2}+p)^{p-1}\\
&\times (x-2p^{3}+p^{2})^{p(p^{2}-p-1)}(x-2p^{3}+1)^{p}.
\end{align*}
\end{theorem}
\textbf{Proof.} The structure of the power graph of the group $ \mathbb{Z}_{p}\times \mathbb{Z}_{p^{2}} $ (see \cite{ghorbani}) is
\[
G\cong K_{1}\vee \bigg(\Big(\underbrace{K_{p-1}\cup \dots \cup K_{p-1}}_{p-\text{times}}\Big)\bigcup \Big (K_{p-1}\vee \big(\underbrace{K_{p^{2}-p} \cup \dots \cup K_{p^{2}-p}}_{p-\text{times}}  \big)\Big)\bigg). \]
From the structure of $ G $, it is clear that the vertices of $ K_{p-1} $ (the one in the second part of the main union) share the same neighbourhood. 
So, the common transmission of the vertices of $ K_{p-1} $ is $ p-2+1+p(p^{2}-p)+2p(p-1)=p^{3}+p^{2}-p-1 $, and by Lemma \ref{eigenvalue of clique}, we see that $ \big(x-p^{3}-p^{2}+p\big) $ is the factor of the $ D^{L} $ characteristic polynomial of $ G $ with algebraic multiplicity $ p-2. $
Also the vertices of each $ K_{p-1} $  in $ \underbrace{K_{p-1}\cup  \dots \cup K_{p-1}}_{p-\text{times}} $ share the same neighbourhood and have common transmission
$$ p-2+2(p-1)^{2}+1+2(p-1)+2p(p^{2}-p)=2p^{3}-p-1. $$
Thus by Lemma \ref{eigenvalue of clique}, we get the factor $ \big( x-2p^{3}+p\big) $ of the $ D^{L} $ characteristic polynomial of $ G $ with algebraic multiplicity $ p(p-2) $. Lastly, the vertices of each $ K_{p^{2}-p} $ in $ \underbrace{K_{p^{2}-p} \cup \dots \cup K_{p^{2}-p}}_{p-\text{times}} $ have the same neighbourhood with common transmission
$$ 1+2p(p-1)+p-1+p^{2}-p-1+2(p-1)(p^{2}-p)=2p^{3}-p^{2}-1. $$
 Therefore, by Lemma \ref{eigenvalue of clique}, we see that $ { \big(x-2p^{3}+p^{2}\big)^{p(p^{2}-p-1)}} $ is the factor of the $ D^{L} $ characteristic polynomial of $ G. $ In this way, we have obtained $ p(p-2)+p-2+p(p^{2}-p-1)={ p^{3}-2p-2} $ integer zeros of the $ D^{L} $ characteristic polynomial of $ G. $ The other $ 2p+2 $ zeros of the $ D^{L} $ characteristic polynomial of $ G $ can be found by using Equation \eqref{eigen equation}. By above observations of common neighbourhoods of vertices of $ G $ and by Lemma \ref{lemma equal componenet}, we take the eigenvector $ X\in \mathbb{R}^{n} $ equal to
\begin{align*}
 X=\Big(x_{1}, &\underbrace{x_{2},x_{2},\dots, x_{2}}_{{ p-1}},\underbrace{x_{3},x_{3},\dots, x_{3}}_{p-1}, \underbrace{x_{4},x_{4},\dots,x_{4}}_{p-1},\dots,\underbrace{x_{p+2},x_{p+2},\dots, x_{p+2}}_{p-1},\\
&\underbrace{x_{p+3},x_{p+3},\dots, x_{p+3}}_{p^{2}-p}, \underbrace{x_{p+4},x_{p+4},\dots,x_{p+4}}_{p^{2}-p},\dots,\underbrace{x_{2p+2},x_{2p+2},\dots, x_{2p+2}}_{p^{2}-p}
\Big),
\end{align*}
and by using Equation \eqref{eigen equation}, we have
\begin{align*}
\partial x_{1}=&~(p^{3}-1)x_{1}-(p-1)x_{2}-(p-1)x_{3}-(p-1)x_{4}-\dots-(p-1)x_{p+2}-(p^{2}-p)x_{p+3}\\
&~-(p^{2}-p)x_{p+4}-\dots-(p^{2}-p)x_{2p+2},\\
\partial x_{2}=&~-x_{1}+(p^{3}+p^{2}-2p+1)x_{2}-2(p-1)x_{3}-2(p-1)x_{4}-\dots-2(p-1)x_{p+2}\\
&~-(p^{2}-p)x_{p+3}-(p^{2}-p)x_{p+4}-\dots-(p^{2}-p)x_{2p+2},\\
\partial x_{3}=&~-x_{1}-2(p-1)x_{2}+(2p^{3}-p-1)x_{3}-(p-2)x_{3}-2(p-1)x_{4}-2(p-1)x_{5}-\dots\\
&~-2(p-1)x_{p+2}-2(p^{2}-p)x_{p+3}-2(p^{2}-p)x_{p+4}-\dots-2(p^{2}-p)x_{2p+2},\\
\partial x_{4}=&~-x_{1}-2(p-1)x_{2}-2(p-1)x_{3}+(2p^{3}-p-1)x_{4}-(p-2)x_{4}-2(p-1)x_{5}-\dots\\
&~-2(p-1)x_{p+2}-2(p^{2}-p)x_{p+3}-2(p^{2}-p)x_{p+4}-\dots-2(p^{2}-p)x_{2p+2},\\
& ~ \vdots\\
\partial x_{p+2}=&~-x_{1}-2(p-1)x_{2}-2(p-1)x_{3}-2(p-1)x_{4}-\dots+(2p^{3}-p-1)x_{p+2}-(p-2)x_{p+2}\\
&~-2(p^{2}-p)x_{p+3}-2(p^{2}-p)x_{p+4}-\dots-2(p^{2}-p)x_{2p+2},\\
\partial x_{p+3}=&~-x_{1}-(p-1)x_{2}-2(p-1)x_{3}-2(p-1)x_{4}-\dots-2(p-1)x_{p+2}+(2p^{3}-p^{2}-1)x_{p+3}\\
&~-(p^{2}-p-1)x_{p+3}-2(p^{2}-p)x_{p+4}-\dots-2(p^{2}-p)x_{2p+2},\\
& ~ \vdots\\
\partial x_{2p+2}=&~-x_{1}-(p-1)x_{2}-2(p-1)x_{3}-2(p-1)x_{4}-\dots-2(p-1)x_{p+2}-2(p^{2}-p)x_{p+3}\\
&~-2(p^{2}-p)x_{p+4}-\dots-2(p^{2}-p)x_{2p+1}+(2p^{3}-p^{2}-1)x_{2p+2}-(p^{2}-p-1)x_{2p+2}.
\end{align*}
The coefficient matrix of the right side of the above system of equations is

\begin{footnotesize}
\begin{equation*}
M=\left (\begin{array}{c | c | c c c | c c c }
d_{11} & -(p-1) & -(p-1) & \cdots & -(p-1) & -(p^{2}-p) & \cdots & -(p^{2}-p)\\
 \hline
-1 & d_{22} & -2(p-1) & \cdots & -2(p-1) & -(p^{2}-p) & \cdots & -(p^{2}-p)\\
\hline
-1 & -2(p-1) & 2p^{3}-2p+1 & \cdots & -2(p-1) & -2(p^{2}-p) & \cdots & -2(p^{2}-p)\\
\vdots & \vdots & \vdots & \ddots & \vdots & \vdots & \ddots & \vdots\\
-1 & -2(p-1) & -2(p-1) & \cdots & 2p^{3}-2p+1 & -2(p^{2}-p) & \cdots & -2(p^{2}-p)\\
\hline
-1 & -(p-1) & -2(p-1) & \cdots & -2(p-1) & 2p^{3}-2p^{2}+p & \cdots & -2(p^{2}-p)\\
\vdots & \vdots & \vdots & \ddots & \vdots & \vdots & \ddots & \vdots\\
-1 & -(p-1) & -2(p-1) & \cdots & -2(p-1) & -2(p^{2}-p) & \cdots & 2p^{3}-2p^{2}+p\\
\end{array}
\right ),
\end{equation*}
\end{footnotesize}
where $ d_{11}=p^{3}-1$ and $ d_{22}=p^{3}+p^{2}-2p+1. $

For, $ i=2,\dots,p $, define
\[ X_{i-1}=\big(0,0,-1, x_{i2}, \dots,x_{ip},\underbrace{ 0,\dots,0}_{p-\text{times}}\big), \]
{ 
where $ x_{ij}=\delta_{ij}$ and $\delta_{ij}$ is the Kronecker delta. }

Now, it can be easily verified that $ X_{1}, \dots, X_{p-1} $ are the linearly independent vectors.
Also \[  MX_{1}=\Big(0, 0, -2p^{3}+1,2p^{3}-1, 0, \dots, 0  \Big)=\big(2p^{3}-1\big)X_{1}. \]
So, it follows that $ X_{1} $ is the eigenvector corresponding to the eigenvalue $ 2p^{3}-1 $ of $ M. $ Similarly, $ X_{2},\dots, X_{p-1} $ are the eigenvectors of the eigenvalue $ 2p^{3}-1 $ of $ M. $ Thus, $ { \big(x-2p^{3}+1\big)^{p-1}} $ is the factor of the $ D^{L} $ characteristic polynomial of $ G. $

Again, for $ i=2,\dots,p $, let
\[ Y_{i-1}=\big(\underbrace{0,\dots,0}_{p+2}, 1,y_{i2}, \dots, y_{ip} \big), \]
{  where $ y_{ij}=-\delta_{ij}$ and $\delta_{ij}$ is the Kronecker delta. }

As above, it can be verified that $ Y_{1},\dots, Y_{p-1} $ are the eigenvectors of $ M $ corresponding to the eigenvalue $ { 2p^{3}-p}. $ Therefore, $ { (x-2p^{3}+p)^{p-1}} $ is the another factor of the $ D^{L} $ characteristic polynomial of $ G. $ Till now, we have obtained $ n-4 $ zeros of the $ D^{L} $ characteristic polynomial of $ G $. The remaining four zeros of the $ D^{L} $ characteristic polynomial of $ G $ can be found by the concept of equitable quotient matrix.
Thus, the remaining four zeros of the $ D^{L} $ characteristic polynomial of $ G $ are the zeros of the characteristic polynomial of the following equitable quotient matrix $ Q $ of the block matrix $ M $
\[
Q=\begin{pmatrix}
p^{3}-1 & -p+1 & -p(p-1) & -p(p^{2}-p)\\
-1 & p^{3}+p^{2}-2p+1 & -2p(p-1) & -p(p^{2}-p)\\
-1 & -2(p-1) & 2p^{3}-2p^{2}+2p-1 & -2p(p^{2}-p)\\
-1 & -p+1 & -2p(p-1) & 2p^{2}-p
\end{pmatrix}.
\]
The characteristic polynomial of $ Q $ is
\[ x(x-p^{3})(x-p^{3}-p^{2}+p)(x-2p^{3}+1). \]
Hence, the $ D^{L} $ characteristic polynomial of $ G $ is
\[ x(x-p^{3})(x-2p^{3}+p)^{p^{2}-p-1}(x-p^{3}-p^{2}+p)^{p-1}(x-2p^{3}+p^{2})^{p(p^{2}-p-1)}(x-2p^{3}+1)^{p}, \]
that proves the result completely. \qed

The following is an immediate consequence of Theorem \ref{theorem p x p^2} and Lemma \ref{diameter 2 lemma}.
\begin{corollary}
The Laplacian characteristic polynomial of $ \mathcal{P}(\mathbb{Z}_{p}\times \mathbb{Z}_{p^{2}}) $ is
\[ x(x-p^{3})(x-p)^{p^{2}-p-1}(x-p^{3}+p^{2}-p)^{p-1}(x-p^{2})^{p(p^{2}-p-1)}(x-1)^{p}. \]
\end{corollary}
\bigskip

Following the structure (see \cite{ghorbani} Theorem 3.3) of
\[ \mathcal{P}(\mathbb{Z}_{p}\times \mathbb{Z}_{p} \times \mathbb{Z}_{p})\cong K_{1}\vee \Big(\underbrace{K_{p-1},K_{p-1},\dots,K_{p-1}}_{p^{2}+p+1} \Big) \]
 and proceeding as in Theorem \ref{theorem p x p^2}, we have the following results.
\begin{theorem}\label{theorem zp x zp x zp}
The distance Laplacian characteristic polynomial of $ \mathcal{P}(\mathbb{Z}_{p}\times \mathbb{Z}_{p} \times \mathbb{Z}_{p}) $ is
\[ x(x-p^{3})(x-2p^{3}+1)^{p^{2}+p}(x-2p^{3}+p)^{p^{3}-p^{2}-p-2}. \]
\end{theorem}

The next result follows immediately from the above theorem and Lemma \ref{diameter 2 lemma}.
\begin{corollary}
The Laplacian characteristic polynomial of $ \mathcal{P}(\mathbb{Z}_{p}\times \mathbb{Z}_{p} \times \mathbb{Z}_{p}) $ is
\[ x(x-p^{3})(x-1)^{p^{2}+p}(x-p)^{p^{3}-p^{2}-p-2}. \]
\end{corollary}
Let $ G\cong \langle x,y : x^{p^{2}}=y^{p}=1, y^{-1}xy=x^{p+1}\rangle. $
Then for $ p\neq 2 $  (see \cite[Theorem 3.5]{ghorbani} )
\[
 \mathcal{P}(G)\cong \mathcal{P}(\mathbb{Z}_{p}\rtimes \mathbb{Z}_{p^{2}}) \cong \mathcal{P}(\mathbb{Z}_{p}\times \mathbb{Z}_{p^{2}})
\]
and its $ D^{L} $ characteristic polynomial is given by Theorem \ref{theorem p x p^2}. For $ p=2 $,
\[
 \mathcal{P}(G)\cong \mathcal{P}(\mathbb{Z}_{2}\rtimes \mathbb{Z}_{2^{2}})=K_{1}\vee (K_{3}\cup \overline{K}_{4})
\]
  and its $ D^{L} $ characteristic polynomial is
\begin{equation}\label{z2 x z4}
 x(x-8)(x-12)^{2}(x-15)^{4}.
\end{equation}
Let
$ G\cong \langle x,y,z : x^{p}=y^{p}=z^{p}=1, xy=yx, zy=yz, xz=zxy \rangle. $ Then, for $ p=2 $,
\[
\mathcal{P}(G)\cong \mathcal{P}(\mathbb{Z}_{p}\rtimes (\mathbb{Z}_{p}\times \mathbb{Z}_{p}))\cong \mathcal{P}(\mathbb{Z}_{p}\times \mathbb{Z}_{p}\times \mathbb{Z}_{p})
\]
 and its $ D^{L} $ characteristic polynomial is given in Theorem \ref{theorem zp x zp x zp}. For $ p=2 $,
\[
\mathcal{P}(G)\cong \mathcal{P}(\mathbb{Z}_{2}\rtimes (\mathbb{Z}_{2}\times \mathbb{Z}_{2}))\cong \mathcal{P}(\mathbb{Z}_{2}\rtimes \mathbb{Z}_{4})
\]
 and the $ D^{L} $ characteristic polynomial is given by \eqref{z2 x z4}.

The $ D^{L} $ characteristic polynomial of the power graph of the cyclic group $ \mathbb{Z}_{pqr} $ is given in \cite{bilalacta}.

Next, we find the $ D^{L} $ characteristic polynomial of $ \mathcal{P}(\mathbb{Z}_{r} \times F_{p,q}), ~ (p\equiv 1 \pmod {q})$. The $ D^{L} $ characteristic polynomial of   $ \mathcal{P}(\mathbb{Z}_{q} \times F_{p,r}), ~ (p\equiv 1 \pmod {r}))$ and $ \mathcal{P}(\mathbb{Z}_{p} \times F_{q,r}), ~ (q\equiv 1 \pmod {r}))$ can be similarly found. The power graph of the group $ \mathcal{P}(\mathbb{Z}_{r} \times F_{p,q}), ~ (p\equiv 1 \pmod {q}))$ (see \cite{ghorbani}, Theorem 3.10) as a joined union is shown in Figure \ref{fig 11}, where $ K_{qr-q-r+1} $ and $ K_{q-1} $ both occur $ p $-times.

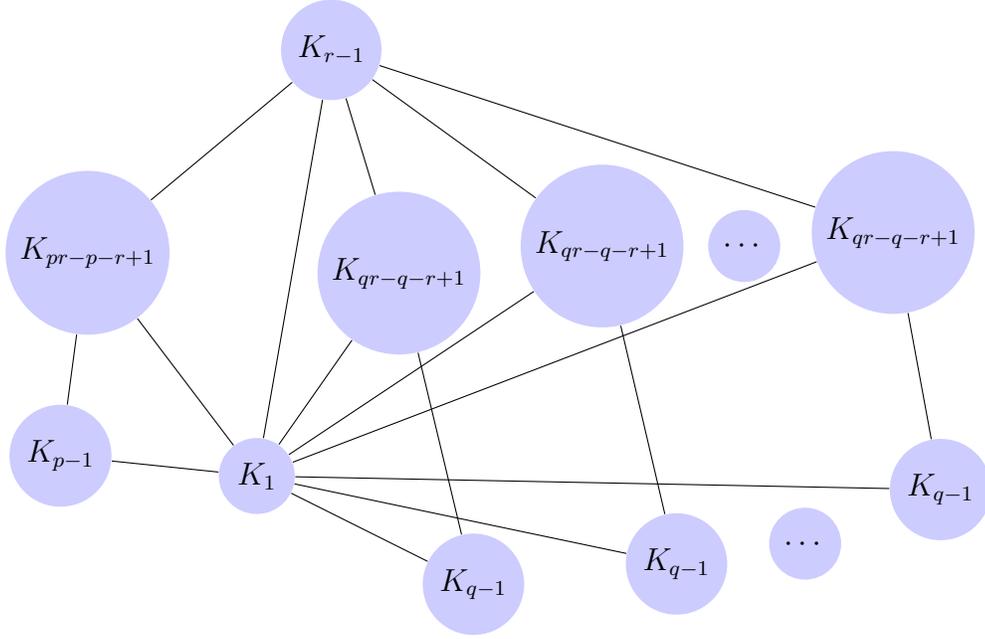
\begin{figure}[H]
\centering
\begin{tikzpicture}
[scale=.9, auto=left, every node/.style={circle, fill=blue!20, minimum width=0.1cm}]
\node (n1) at (0.4,0.7) {$ K_{1} $}; \node (1) at (-2.5,1) {$ K_{p-1} $};
\node (2) at (-2.1,4) {$ K_{pr-p-r+1} $}; \node (3) at (1.5,7) {$ K_{r-1} $};
\node (4) at (2.5,3.7) {$ K_{qr-q-r+1} $}; \node (5) at (5.5,4.1) {$ K_{qr-q-r+1} $};
\node (6a) at (7.6,4.1) {$ \cdots $};
\node (6) at (9.8,4.3) {$ K_{qr-q-r+1} $};

\node (7) at (3.6,-.9) {$ K_{q-1} $}; \node (8) at (6.6,-0.6) {$ K_{q-1} $};
\node (8a) at (8.5,-0.3) {$ \cdots $};
\node (9) at (10.5,0.5) {$ K_{q-1} $};

\foreach \from/\to in {n1/1,n1/2,n1/3, 1/2, 2/3, 3/4, 3/5, 3/6, n1/7, n1/8, n1/9,4/7, 5/8, 6/9, n1/4, n1/5, n1/6}
\draw(\from)--(\to);
\end{tikzpicture}

\caption{$ \mathcal{P}(\mathbb{Z}_{r} \times F_{p,q}), ~ (p\cong 1 (\text{mod} q))$. }
\label{fig 11}
\end{figure}

Next, we find the distance Laplacian characteristic polynomial of $\mathcal{P}(\mathbb{Z}_{r}\times F_{p,q}).$
\begin{theorem}\label{theorem zr x Fp,q}
The distance Laplacian characteristic polynomial of $ \mathcal{P}(\mathbb{Z}_{r}\times F_{p,q}) $ is
\begin{align*}
&(x-2pqr+pr-r+1)^{p-2}(x-2pqr+pr)^{pr-p-r}(x-pqr-pq+1)^{r-2}\\
&\times(x-2pqr+qr)^{p(qr-r-q)}(x-2pqr+qr-r+1)^{q-2} (x-\alpha_{1})^{p-1}(x-\alpha_{2})^{p-1} \Theta (x,M),
\end{align*}
where $ \Theta (x,M) $  is the characteristic polynomial of Matrix \eqref{qmat eq 2 for statement} and  
\[
\begin{split}
\alpha_{1}&=2pqr-\textstyle\frac{1}{2}\Big ( qr+1-\sqrt{r^{2}(q-2)^{2}+6qr-8r-4q+5} \Big);\\ \alpha_{2}&=2pqr-\textstyle\frac{1}{2}\Big ( qr+1+\sqrt{r^{2}(q-2)^{2}+6qr-8r-4q+5} \Big).
\end{split}
\]
\end{theorem}
\textbf{Proof.} From Figure $ 1, $ the transmission of the vertex of $ K_{1} $ is $ T_{1}=pqr-1 $, the common transmission of every vertex of $ K_{p-1} $ is
\[
\begin{split}
 T_{2}&=1+p-2+pr-p-r+1+2(r-1)+2p(qr-r-q+1)+2p(q-1)\\
 &=2pqr-pr+r-2,
\end{split}
\]
 the common transmission of each vertex of $ K_{pr-p-r+1} $ is
\[
\begin{split}
T_{3}&=p-1+1+pr-p-r+r-1+2p(qr-r-q+1)+2p(q-1)\\
&=2pqr-pr-1,
\end{split}
\]
the common transmission of each vertex of $ K_{r-1} $ is
\[
\begin{split}
 T_{4}&=r-2+pr-p-r+1+2(p-1)+1+p(qr-r-q+1)+2p(q-1)\\
 &=pqr+pq-2,
\end{split}
\]
the common transmission of each vertex of $ K_{qr-r-q+1} $ is $ T_{5}=2pqr-qr-1, $ and the common transmission of every vertex of $ K_{q-1} $ is $ T_{6}=2pqr-qr+r-2. $ From Figure $ 1 $, keeping in view the common neighbourhoods of particular vertices and applying Lemma \ref{eigenvalue of clique}, we see that
\begin{align*}
 &(x-2pqr+pr-r+1)^{p-2}(x-2pqr+pr)^{pr-p-r}(x-pqr-pq+1)^{r-2}\\
 &\times (x-2pqr+qr)^{p(qr-r-q)}(x-2pqr+qr-r+1)^{q-2},
\end{align*}
are the factors of the $ D^{L} $ characteristic polynomial of $ \mathcal{P}(\mathbb{Z}_{r}\times F_{p,q}) $. By Lemma
\ref{lemma equal componenet}, we can choose the eigenvector
\begin{align*}
 X=&\Big (x_{1}, \underbrace{x_{2},\dots,x_{2}}_{p-1-\text{times}},\underbrace{x_{3},\dots,x_{3}}_{pr-p-r+1-\text{times}}, \underbrace{x_{4},\dots,x_{4}}_{r-1-\text{times}}, \underbrace{x_{5},\dots,x_{5}}_{qr-r-q+1-\text{times}}, \underbrace{x_{6},\dots,x_{6}}_{qr-r-q+1-\text{times}},\dots,\underbrace{x_{p+4},\dots,x_{p+4}}_{qr-r-q+1-\text{times}},\\
& \underbrace{x_{p+5},\dots,x_{p+5}}_{q-1-\text{times}},\underbrace{x_{p+6},\dots,x_{p+6}}_{q-1-\text{times}}, \dots, \underbrace{x_{2p+4},\dots,x_{2p+4}}_{q-1-\text{times}} \Big).
\end{align*}
Now, using Equation \eqref{eigen equation} and following the steps in the proof of Theorem \ref{theorem p x p^2}, the coefficient matrix of the eigenequations can be found as
\begin{footnotesize}
\begin{equation}\label{eq mat zr x fp q}
\left (\begin{array}{c|c|c|c|cccc|cccc}
d_{11} & -p+1 & -c & -r+1 & -d  & -d  &\cdots & -d &  -q+1 & -q+1 & \cdots & -q+1\\
\hline
-1     & d_{22} & -c & -2r+2 & -2d  & -2d  & \cdots & -2d & -2q+2 & -2q+2 & \cdots & -2q+2\\
\hline
-1     & -2p+2 & d_{33} & -r+1 & -2d  & -2d  & \cdots & -2d & -2q+2 & -2q+2 & \cdots & -2q+2\\
\hline
-1     & -2p+2 & -c & d_{44} & -d  & -d  & \cdots & -d & -2q+2 & -2q+2 &\cdots & -2q+2\\
\hline
-1     & -2p+2 & -2c & -r+1 & d_{55}  & -2d  & \cdots & -2d & -q+1 & -2q+2 & \cdots & -2q+2\\
-1     & -2p+2 & -2c & -r+1   & -2d & d_{55} & \cdots & -2d & -2q+2 & -q+1 & \cdots & -2q+2\\
\vdots & \vdots & \vdots & \vdots    & \vdots & \vdots  & \ddots & \vdots &\vdots     & \vdots & \ddots & \vdots\\
-1     & -2p+2 & -2c & -r+1 & -2d  & -2d & \cdots & d_{55} & -2q+2 & -2q+2 &\cdots & -q+1\\
\hline
-1     & -2p+2 & -2c & -2r+2 & -d  & -2d & \cdots & -2d & d_{66} & -2q+2 & \cdots & -2q+2\\
-1     & -2p+2 & -2c & -2r+2 & -2d  & -d & \cdots & -2d & -2q+2 & d_{66}& \cdots & -2q+2\\
\vdots & \vdots & \vdots & \vdots    & \vdots  & \vdots & \ddots & \vdots &  \vdots       & \vdots & \ddots & \vdots\\
-1     & -2p+2 & -2c & -2r+2 & -2d  & -2d & \cdots & -d & -2q+2 & -2q+2 & \cdots & d_{66}
\end{array}
\right ),
\end{equation}
\end{footnotesize}
where $ c=(pr-p-r+1)$, $d=(qr-q-r+1)$,  $d_{11}=pqr-1$, $d_{22}=2pqr-pr-p+r$, $d_{33}=2pqr-2pr+p+r-1$, $d_{44}=pqr+pq-r$,  $d_{55}=2pqr-2qr+q+r-1$ and $ d_{66}=2pqr-qr-q+r. $
The other $ 2p+4 $ zeros of the $ D^{L} $ characteristic polynomial of $ \mathcal{P}(\mathbb{Z}_{r}\times F_{p,q}) $ are the zeros of the characteristic polynomial of \eqref{eq mat zr x fp q}. It is easy to see that $ 0 $ and $ pqr $ are the simple eigenvalues of \eqref{eq mat zr x fp q} with associated eigenvectors $ (0,\dots, 0) $ and $ (-d_{11},1,\dots,1). $ By using calculations with the help of Wolfram Mathematica 13, we have seen that
\[
2pqr-\textstyle \frac{1}{2}\Big ( qr+1\pm\sqrt{r^{2}(q-2)^{2}+6qr-8r-4q+5} \Big)
\]
are the eigenvalues of \eqref{eq mat zr x fp q}, each with algebraic multiplicity $ p-1. $ The other $ 4 $ zeros of the $ D^{L} $ characteristic polynomial of $ \mathcal{P}(\mathbb{Z}_{r}\times F_{p,q}) $ are the zeros (other than $ 0 $ and $ pqr $) of the characteristic polynomial
of the following matrix
\begin{footnotesize}
\begin{equation}\label{qmat eq 2 for statement}
 \begin{pmatrix}
l_{11} & -(p-1)   & -(pr-p-r+1)  & -(r-1)   &-p(qr-r-q+1) & -p(q-1)\\
-1 & l_{22} &  -(pr-p-r+1) &  -2(r-1) & -p(qr-r-q+1) &-2p(q-1)\\
-1 & -(p-1)  & l_{33}   & -(r-1)   & -2p(qr-r-q+1) & -2p(q-1)\\
-1 & -2(p-1) & -(pr-p-r+1) & l_{44} & -2p(qr-r-q+1) & -2p(q-1)\\
-1 & -2(p-1) & -2(pr-p-r+1) & -(r-1) & l_{55} & -2p(q-1)\\
-1 & -2(p-1) & -2(pr-p-r+1) & -(r-1) & (1-2p)(q-1)(r-1) & l_{66}\\
\end{pmatrix},
\end{equation}
\end{footnotesize}
where $l_{11}=pqr-1$, $l_{22}=2pqr-pr-p+r$, $l_{33}=2pqr-2pr+p+r-1$, $l_{44}=pqr+pq-r$, $l_{55}=2pq+2pr-2p-q-r+1 $ and $ l_{66}=2pqr-2pq+2p-qr+q+r-2$.  By using software calculations, we see that
\[
\begin{split}
\Theta (x,M)&=-x(x-pqr)\Big(x^{4}+x^{3}(3 - p q - r + p r + q r - 7 p q r)\\
&+x^{2}(1 + p + q - 2 p q + r + p r + q r - 15 p q r - p^2 q r - p q^2 r + 6 p^2 q^2 r - r^2 + 6 p q r^2 \\
&- 5 p^2 q r^2 - 5 p q^2 r^2 + 18 p^2 q^2 r^2)+x(-2 + p + p^2 + 2 q - p q - p^2 q + 5 r - p r - 2 p^2 r\\
& - 3 q r - p q r- 3 p^2 q r - 5 p q^2 r + 8 p^2 q^2 r - 4 r^2 + 2 p r^2 + p^2 r^2 + 3 q r^2 - 4 p q r^2 - 4 p^2 q r^2 \\
&- 3 p q^2 r^2 + 23 p^2 q^2 r^2 + 4 p^3 q^2 r^2 + 4 p^2 q^3 r^2 - 12 p^3 q^3 r^2 + r^3 - p r^3 - q r^3 + 4 p q r^3\\
& - 11 p^2 q^2 r^3 + 8p^3 q^2 r^3 + 8 p^2 q^3 r^3 - 20 p^3 q^3 r^3)p + p^2 + p q - p^2 q + 2 p r - 2 p^2 r 
\\
&+ p q r + p^2 q r - 2 p^3 q r - 3 p q^2 r + p^2 q^2 r + 2 p^3 q^2 r - p r^2 + p^2 r^2 - 4 p q r^2 - 2 p^2 q r^2
\\
& + 4 p^3 q r^2 + 3 p q^2 r^2 + 2 p^3 q^2 r^2 + 6 p^2 q^3 r^2 - 8 p^3 q^3 r^2 + 4 p q r^3- p^2 q r^3 - 2 p^3 q r^3 
\\
& - 3 p q^2 r^3 + 2 p^2 q^2 r^3 + 4 p^3 q^2 r^3 + 2 p^2 q^3 r^3 - 10 p^3 q^3 r^3 - 4 p^4 q^3 r^3 - 4 p^3 q^4 r^3 + 8 p^4 q^4 r^3
\\
&  - p q r^4 + p^2 q r^4 + p q^2 r^4 - 3 p^2 q^2 r^4 + 6 p^3 q^3 r^4  - 4 p^4 q^3 r^4 - 4 p^3 q^4 r^4 + 8 p^4 q^4 r^4\Big).
\end{split}
\]  
\qed

The following consequence of above result gives the following.

\begin{corollary}\label{cor zr x Fp,q}
The Laplacian characteristic polynomial of $ \mathcal{P}(\mathbb{Z}_{r}\times F_{p,q}) $ is
\begin{align*}
x&(x-pqr)(x-pr+r-1)^{p-2}(x-pr)^{pr-p-r}(x-pqr+pq-1)^{r-2}(x-qr)^{p(qr-r-q)}\\
&\textstyle \times(x-qr+r-1)^{q-2} \left (x-\frac{1}{2}\Big ( qr+1-\sqrt{r^{2}(q-2)^{2}+6qr-8r-4q+5} \Big)\right )^{p-1}\\
&\textstyle\times \left (x-\frac{1}{2}\Big ( qr+1+\sqrt{r^{2}(q-2)^{2}+6qr-8r-4q+5} \Big)\right )^{p-1} \psi(x), \end{align*}
 where $ \psi(x) $ is the characteristic polynomial of the following matrix
\begin{scriptsize}
\begin{equation}\label{qmat L new}
\left(
\begin{array}{cccccc}
 p q r-1 & -(p-1) & -(p r-p-r+1) & -(r-1) & -p (q r-q-r+1) & -p (q-1) \\
 -1 & p r-p-r+2 & -(p r-p-r+1) & 0 & 0 & 0 \\
 -1 & -(p-1) & p+r-1 & -(r-1) & 0 & 0 \\
 -1 & 0 & -(p r-p-r+1) & p q r-p q-r+2 & -p (q r-q-r+1) & 0 \\
 -1 & 0 & 0 & -(r-1) & q+r-1 & -(q-1) \\
 -1 & 0 & 0 & 0 & -(q r-q-r+1) & q r-q-r+2 \\
\end{array}
\right).
\end{equation}
\end{scriptsize}
Moreover, 
\[
\begin{split}
\psi (x)&=x^{4}+x^{3}(-3 + p q + r - p r - q r - p q r)+x^2 (p^2 q r^2-p^2 q r+p q^2 r^2-p q^2 r +3 p q r-2 p q\\
&
+p r+p+q r+q-r^2+r+1)+x(-p^2 q^2 r^3+p^2 q^2 r^2-p^2 q r+p^2 q-p^2 r^2+2 p^2 r-p^2\\
&
-p q^2 r^2+p q^2 r-3 p q r+p q+p r^3-2 p r^2+p r-p+q r^3-3 q r^2+3 q r-2 q-r^3+4 r^2\\
&
-5 r +2)+ p^2 q^2 r^4-2 p^2 q^2 r^3+2 p^2 q^2 r^2-p^2 q^2 r-p^2 q r^4+3 p^2 q r^3-4 p^2 q r^2+3 p^2 q r-p^2 q
\\
&
+p^2 r^2-2 p^2 r+p^2-p q^2 r^4+3 p q^2 r^3-3 p q^2 r^2+p q^2 r+p q r^4-4 p q r^3\\&
+6 p q r^2-3 p q r+p q-p r^2+2 p r-p .
\end{split}
\]
\end{corollary}

Next, we have the $ D^{L} $ Characteristic polynomial of the power graph of the group $ F_{p,qr}, ~ \big (p\equiv 1\pmod {qr} \big).$ The structure of $ \mathcal{P}(F_{p,qr}) $ (\cite{ghorbani}, Theorem 3.12) is given by
\begin{align*}
\mathcal{P}(F_{p,qr})=\begin{cases}
K_{1}\vee \Big(K_{p-1}\vee K_{pr-p-r+1}\vee K_{r-1} \vee \big( \underbrace{K_{qr-r}\cup \dots \cup K_{qr-r}}_{p}\big)\Big) & \text{if}~  r=3~  $ or $ q=3,\\
K_{1}\vee \Big(\underbrace{K_{qr-1}\cup K_{qr-1}\cup \dots \cup K_{qr-1}}_{p} \Big) & $ if $ q,r\neq 3.
\end{cases}
\end{align*}
 Proofs of the following results can be worked out as in Theorems \ref{theorem p x p^2} and \ref{theorem zr x Fp,q}.
\begin{theorem}\label{theorem F p, qr}
Let $ G\cong F_{p,qr}$, with $p\equiv 1 \pmod {qr})$. Then the following holds:
\begin{itemize}
\item[\bf (i)]  For $ r=3 $ or $ q=3,$ the distance Laplacian characteristic polynomial of $ \mathcal{P}(F_{p,qr}) $ is
\begin{align*}
x&(x-pqr)(x-2pqr+pr-r+1)^{p-2}(x-2pqr+pr)^{pr-p-r}(x-pqr-p+1)^{r-2}\\
&\times(x-2pqr+qr)^{p(qr-r-1)}(x-2pqr-r)^{p-1}g(x),
\end{align*}
where 
\[
\begin{split}
g(x)&=x^{3}+x^{2}(2-p+pr-5pqr)+x (8 p^2 q^2 r^2-3 p^2 q r^2+4 p^2 q r-p^2 r-7 p q r\\& 
+p r^2 -r^2+2 r) -4 p^3 q^3 r^3+2 p^3 q^2 r^3-4 p^3 q^2 r^2+6 p^2 q^2 r^2+2 p^3 q r^2-p^2 q r^3\\&
-p^2 q r^2+p^2 q r-p^2 r+p q r^3-2 p q r^2-p q r+p r.
\end{split}
\]
\item[\bf (ii)] For $ r,q\neq 3 $, the distance Laplacian characteristic polynomial of $ \mathcal{P}(F_{p,qr}) $ is
\begin{align*}
x&(x-pqr)(x-2pqr+pr-r+1)^{p-2}(x-2pqr+pr)^{pr-p-r}(x-pqr-p+1)^{r-2}\\
&\times (x-2pqr+qr)^{p(qr-r-1)}(x-2pqr+r)^{p-1}.
\end{align*}
\end{itemize}
\end{theorem}
\begin{corollary}\label{cor F p, qr}
Let $ G\cong F_{p,qr}$, with $p\equiv  1\pmod{qr}$. Then the following holds:
\begin{itemize}
\item[\bf (i)]  For $ r=3 $ or $ q=3,$ the Laplacian characteristic polynomial of $ \mathcal{P}(F_{p,qr}) $ is
\begin{align*}
x&(x-pqr)(x-pr+r-1)^{p-2}(x-pr)^{pr-p-r}(x-pqr+p-1)^{r-2}(x-qr)^{p(qr-r-1)}\\
&\times(x-r)^{p-1}h(x),
\end{align*}
where 
\[
\begin{split} 
h(x)&=x^{3}+x^{2}(-2 + p - p r - p q r)+x(2 r - p^2 r + p q r - r^2 + p r^2 + p^2 q r^2)-p r \\&
+ p^2 r + p q r - p^2 q r - 2 p q r^2 + p^2 q r^2 + p q r^3 - p^2 q r^3.
\end{split}
\]

\item[\bf (ii)] For $ r,q\neq 3 $, the Laplacian characteristic polynomial of $ \mathcal{P}(F_{p,qr}) $ is
\begin{align*}
x&(x-pqr)(x-pr+r-1)^{p-2}(x-pr)^{pr-p-r}(x-pqr+p-1)^{r-2}\\
&\times(x-qr)^{p(qr-r-1)}(x-r)^{p-1}.
\end{align*}
\end{itemize}
\end{corollary}

The power graph of $ G_{i+5} $ (\cite{ghorbani1}, Theorem 3.14) is
\[ \mathcal{P}(G_{i+5})= K_{1}\vee \Big( \big(\underbrace{K_{r-1}\cup\dots \cup K_{r-1}}_{p}\big)\bigcup \big(K_{p-1}\vee K_{pr-p-r+1}\vee K_{q-1}\big)\Big),\] and its (distance) Laplacian polynomials are given by the following result.
\begin{theorem}\label{theorem Gi+5}
Let $ G\cong G_{i+5}$. Then the following holds:
\begin{itemize}
\item[\bf (i)] The distance Laplacian characteristic polynomial of $ \mathcal{P}(G_{i+5}) $ is
\begin{align*}
x(x-&pqr)(x-2pqr+pq-p-q+2)(x-2pqr+1)^{pq}(x-2pqr+pq-q+1)^{p-2}\\
&\times(x-2pqr+pq)^{pq-p-q+1}(x-2pqr+pq-p+1)^{q-2}(x-2pqr+2r-2)^{pq(r-2)}.
\end{align*}
\item[\bf (ii)] The Laplacian characteristic polynomial of $ \mathcal{P}(G_{i+5}) $ is
\begin{align*}
x(x-&pqr)(x-pq+p+q-2)(x-1)^{pq}(x-pq+q-1)^{p-2}(x-pq)^{pq-p-q+1}\\
&\times(x-pq+p-1)^{q-2}(x-2r+2)^{pq(r-2)}.
\end{align*}
\end{itemize}
\end{theorem}

\section{Distance Laplacian characteristic polynomials of proper power graphs}\label{section 3}
\paragraph{} For further analysis, various types of power graphs of groups are studied like reduced power graph, proper power graph, enhanced power graph and many others. If the identity $ e $ is removed from group $ \mathcal{G} $, the corresponding power graph $ \mathcal{P}(\mathcal{G}^{*}) $ is called proper power graph.

The dicyclic groups of order $ 4n $ are denoted and presented as follows
\begin{align*}
Q_{n}&=\langle a,b ~|~ a^{2n}=e, b^2=a^n, ab=ba^{-1}\rangle.
\end{align*}
If  $Q_{n}$ is a $2$-group, then its  called the \emph{generalized quaternion group} of order $ 4n$ and
\[
\mathcal{P}(Q_{n}^{*})\cong K_{1}\vee \big( K_{2n-2}, \underbrace{K_{2}, \dots,K_{2}}_{n}\big). \]
The next results, gives the structure of proper power graphs of a group $ \mathcal{G}. $
\begin{theorem}[\cite{cameron2020}, Theorem 4]\label{cam th1}
Let $ \mathcal{G} $ be a  group of order $n$. The set of vertices which are joined to all other vertices in $ \mathcal{P}(\mathcal{G}) $ is
\begin{itemize}
\item[\bf (i)] $ \mathcal{G}, $ if $ \mathcal{G} $ is cyclic of prime power order.
\item[\bf (ii)] The set of generators of $ \mathcal{G} $ together with the identity, if $ \mathcal{G} $ is cyclic but not of prime power order.
\item[\bf (iii)] $ Z(\mathcal{G}) $, if $ \mathcal{G} $ is a generalized quaternion group, where $ Z(\mathcal{G}) $ is the center of $ \mathcal{G}. $
\item[\bf (iv)] $ \{e\} $, in any other case.
\end{itemize}
\end{theorem}
The following is an immediate consequence of Theorem \ref{cam th1}.
\begin{corollary}\label{cor cam th 1}
Let $ \mathcal{G} $ be a group of order $n$. The set of vertices which are joined to all other vertices in $ \mathcal{P}(\mathcal{G}^{*}) $ is
\begin{itemize}
\item[\bf (i)] $ \mathcal{G}^{*}, $ if $ \mathcal{G} $ is cyclic of prime power order.
\item[\bf (ii)] The set of generators of $ \mathcal{G}^{*} $, if $ \mathcal{G} $ is cyclic but not of prime power order.
\item[\bf (iii)] $ Z(\mathcal{G})-e=\{a^{n}\} $, if $ \mathcal{G} $ is a generalized quaternion group, where $ Z(\mathcal{G}) $ is the center of $ \mathcal{G}. $
\item[\bf (iv)] Otherwise $  \mathcal{P}(\mathcal{G}^{*})  $ is disconnected.
\end{itemize}
\end{corollary}
\textbf{Proof} By the definition of power graph and Theorem \ref{cam th1}, the proof follows. \qed
From Corollary \ref{cor cam th 1}, we see that the proper power graphs of $ \mathbb{Z}_{n} $ and the generalized quaternion group are connected, so the distance Laplacian characteristic polynomial makes sense.

An integer $ d $ is a proper divisor of $ n $ if $ d $ divides $ n $ (written as $ d|n $) and $ 1<d<n .$ Let $ d_{1}, \dots,d_{t} $ be the distinct proper divisors of $ n. $ Let $\Delta_{n}$ be a simple graph with vertex set $ \{d_{1}, \dots,d_{t}\} $ in which two distinct vertices are adjacent (joined by an edge) if and only if $ d_{i}|d_{j} $, for $ 1\leq i<j\leq t $. If  $ n=p_{1}^{n_{1}}\cdots p_{r}^{n_{r}} $, where $ r,n_{1},\dots,n_{r} $ are positive integers and $ p_{1}, \dots,p_{r} $ are distinct prime numbers, then it is easy to see that the size of $\Delta_{n}$ is given by $ |V( \Delta_{n})|=\prod\limits_{i=1}^{r}(n_{i}+1)-2. $

The  power graph of the  group $ \mathbb{Z}_{n} $ can be written as join of the certain complete graphs.
\begin{lemma}[\cite{mehreen, mehreen1}]\label{mehreen}
If $ \mathbb{Z}_{n} $ is a cyclic group of order $n$, then
\begin{equation*}
\mathcal{P}^{*}(\mathbb{Z}_{n}) =K_{\phi(n)}\vee \Delta_{n}[K_{\phi(d_{1})},\dots,K_{\phi(d_{t})}],
\end{equation*}
where $ \phi(n) $ is the Eulers totient function.
\end{lemma}

The following result gives the distance Laplacian characteristic polynomial of reduced power graph of $ \mathbb{Z}_{n}. $
\begin{theorem}\label{ proper power graph th1}
The distance Laplacian characteristic polynomial of $ G\cong \mathcal{P}^{*}(\mathbb{Z}_{n}) $ is
\begin{align*}
\Theta(G,x)=x(x-n+1)^{\phi(n)-1}\big(x-(\phi(d_{i})+\phi(n)+\beta_{i})\big)^{\phi(d_{i})-1}\Theta(M,x),
\end{align*}
where, for $ i=1,\dots,t $~ $ \beta_{i}=\phi(n)+\sum\limits_{k=2, k\neq 1,i}^{t}n_{k}d_{\Delta_{n}}(v_{i},v_{k}) $ and
$ \Theta(M,x) $ is the characteristic polynomial of the following matrix

\begin{equation}\label{qmat proper power graph}
\begin{footnotesize}
M=\begin{pmatrix}
\beta_1                       & -\phi(d_{1})                    & \dots  & -\phi(d_{n})\\
-(\phi(n)+1) d(v_{2},v_{1})   & \beta_2                         & \dots  & -\phi(d_{t}) d(v_{2},v_{t+1})\\
\vdots                        & \vdots                          & \ddots & \vdots\\
-(\phi(n)+1) d(v_{t+1},v_{1}) & -(\phi(d_{1})) d(v_{t+1},v_{2}) & \dots  & \beta_{t+1}
\end{pmatrix}
\end{footnotesize}.
\end{equation}
\end{theorem}
\textbf{Proof.} By Lemma \ref{mehreen}, we see that
$ G\cong K_{\phi(n)}\vee \Delta_{n}[K_{\phi(d_{1})},K_{\phi(d_{2})},\dots,K_{\phi(d_{t})}]. $ Clearly, each  vertex of $ K_{\phi(n)} $ is connected to every vertex of $ G $ with common transmission
\[
\phi(n)-1+\sum\limits_{i=1}^{n}\phi(d_{i})=\phi(n)-1+n-\phi(n)-1=n-2,
\]
as $ \sum\limits_{d|n}\phi(d)=n. $ Thus, by { Lemma \ref{eigenvalue of clique}}, we see that $ { (x-n+1)} $ is the factor of $ \Theta(G,x) $ with multiplicity at least $ \phi(n)-1. $ Similarly, the common transmission of each vertex in $ K_{\phi(d_{i})} $ is
\[
\begin{split}
 \phi(n)+\phi(d_{i})-1+&\phi(d_{1})d(v_{i},v_{1})+\phi(d_{2})d(v_{i},v_{2})+\dots+\phi(d_{i-1})d(v_{i},v_{i-1})\\
+&\phi(d_{i+1})d(v_{i},v_{i+1})+\dots+\phi(d_{t})d(v_{i},v_{t})\\
=&\phi(n)+\phi(d_{i})-1+\beta_{i},\qquad (i=1,\dots,t)
\end{split}
\]
$ \beta_{i}=\sum\limits_{k\neq i, k=1}^{t}\phi(d_{k})d(v_{i},v_{k}) $ and $ v_{1}, \dots,v_{t} $ are the vertices of $ \Delta_{n}. $ Therefore, $ x-(\phi(n)+\phi(d_{i})+\beta_{i}) $ is the factor of $ \Theta(G,x)  $ with multiplicity $ \phi(d_{i})-1 $ for each $ i$ by  Lemma \ref{eigenvalue of independence number}.  In this way, we have obtained $ \phi(n)-1+\sum\limits_{i=1}^{t}\big(\phi(d_{i})-1\big)=\phi(n)-1+n-\phi(n)-1-t=n-t-2 $ zeros of $ \Theta(G,x) $. The other $ t-2 $ zeros of $ \Theta(G,x) $ are the zeros of the characteristic polynomial of Matrix \eqref{qmat proper power graph}. \qed

Finally, we find the distance Laplacian characteristic polynomial of the power graphs of the generalized quaternion group.
The followings results gives the Laplacian and the distance Laplacian characteristic polynomial of $ \mathcal{P}(Q_{n}^{*}) $, whose proofs are similar to the proofs of above results.
\begin{theorem}\label{dicyclic} Let $ \mathcal{P}(Q_{n}^{*}) $ be a proper power graph of $2$-group $ Q_{n}$. Then the following holds.
\begin{itemize}
\item[\bf (i)] The distance Laplacian characteristic polynomial of $ \mathcal{P}(Q_{n}^{*}) $ is
\[ x(x-4n+1)(x-6n+1)^{2n-3}(x-8n+5)^{n}(x-8n+3)^{n}. \]
\item[\bf (ii)] The Laplacian characteristic polynomial of $ \mathcal{P}(Q_{n}^{*}) $ is
\[ x(x-3)^{n}(x-1)^{n}(x-4n+1)(x-2n+1)^{2n-3}. \]
\end{itemize}
\end{theorem}

\section{Comments}\paragraph{}
In this article, the characteristic polynomials of the Laplacian matrix and the distance Laplacian matrix of the power graphs of groups of order $ pqr $ and the proper power graphs of the cyclic and the generalized quaternions groups are found. Further, from Theorems \ref{theorem p x p^2}, \ref{theorem zp x zp x zp} and \ref{theorem Gi+5}, the power graphs of $ \mathbb{Z}_{p}\times \mathbb{Z}_{p^{2}}, ~ \mathbb{Z}_{p}\times\mathbb{Z}_{p}\times \mathbb{Z}_{p} $ and $ G_{i+5} $ are Laplacian (distance) integral, that is, zeros of their corresponding characteristic polynomials are integers.
Besides by Theorem \ref{cam th1} and Theorem $ 3.4 $ of \cite{AH2}, (see also \cite{bilalacta}) the following holds
\begin{equation}\label{spectral decrease}
 \partial_{i}^{L}(\mathcal{P}(\mathcal{G}))\geq \partial_{i}^{L}(\mathcal{P}(\mathbb{Z}_{p^{m}})),\qquad
 \qquad  (i=1,\dots,n),
\end{equation}
where $ p $ is prime and $ m $ is a positive integer. The following observations are immediate from \eqref{spectral decrease} and the fact that $ \partial_{n-1}^{L}(\mathcal{P}(\mathcal{G}))\geq n $ (see \cite{AH2}, Theorem 4.2, \cite{bilalacta}), with equality if and only if $ \overline{\mathcal{P}(\mathcal{G})} $ is disconnected.
\begin{itemize}
\item $ \partial_{n-1}^{L}=\mu_{1}=n $, for any power graph $ \mathcal{P}(\mathcal{G}) $ of finite group $ \mathcal{G}. $
\item $ \partial_{i}^{L}(\mathcal{P}(\mathcal{G}))\geq \partial_{i}^{L}(\mathcal{P}(\mathbb{Z}_{p^{m}}))=n $ for $ i=1,2,\dots,n-1 $ and $ \partial_{n}^{L}(\mathcal{P}(\mathcal{G}))= \partial_{n}^{L}(\mathcal{P}(\mathbb{Z}_{p^{m}}))=0. $

\item $ \partial^{L}_{n-i}(\mathcal{P}(\mathbb{Z}_{n}))=n $ for $ i=1,2,\dots, \phi(n)+1 $, $ \partial^{L}_{n-i}(\mathcal{P}(\mathbb{Z}_{n}))=n+\phi(p) $ for $ i=1,2,\dots, \phi(q)-1 $, $ \partial^{L}_{n-i}(\mathcal{P}(\mathbb{Z}_{n}))=n+\phi(q) $ for $ i=1,2,\dots, \phi(p)-1 $, $ \partial^{L}_{1}(\mathcal{P}(\mathbb{Z}_{n}))=n+\phi(p)+\phi(q) $ and $ \partial^{L}_{n}(\mathcal{P}(\mathbb{Z}_{n}))=0 $, provided $ n $ is product of two primes $ n=pq, ~ (p<q) $.
\item $ \partial^{L}_{n-i}(\mathcal{P}(\mathbb{Z}_{n}))=n $ for $ i\leq \phi(n)+1 $ and $ \partial^{L}_{n-i}(\mathcal{P}(\mathbb{Z}_{n}))\geq n $ for $ \phi(n)+1< n \leq n-1, $  provided $ n $ is neither prime power nor product of two distinct primes.
\item $ \partial^{L}_{2n-1}(\mathcal{P}(D_{2n}))=2n, ~ \partial^{L}_{2n-1-i}(\mathcal{P}(D_{2n}))=3n $ for $ i=1,2,\dots, n-2 $, $ \partial^{L}_{n+1-i}(\mathcal{P}(D_{2n}))=4n-1 $ for $ i=1,2,\dots,n $ and $ \partial^{L}_{2n}(\mathcal{P}(D_{2n}))=0, $ provided $ n $ is a prime power, where $ D_{2n} $ is the dihedral group.
\item $ \partial^{L}_{2n-1}(\mathcal{P}(D_{2n}))=2n, ~ \partial^{L}_{2n-1-i}(\mathcal{P}(D_{2n}))=3n $, for $ i=1,2,\dots, \phi(n) $, $ \partial^{L}_{2n-1-\phi(n)-i}(\mathcal{P}(D_{2n}))=3n+\phi(p) $, for $ i=1,2,\dots, \phi(q)-1 $, $ \partial^{L}_{2n-\phi(n)-\phi(q)-i}(\mathcal{P}(D_{2n}))=3n+\phi(q), $ for $ i=1,\dots,\phi(p)-1 $, $ \partial^{L}_{n+2-i}(\mathcal{P}(D_{2n}))=4n-\phi(n)-1, $ $ \partial^{L}_{n+1-i}(\mathcal{P}(D_{2n}))=4n-1, $ for $ i=1,2,\dots, n $ and $ \partial^{L}_{2n}(\mathcal{P}(D_{2n}))=0, $ provided $ n $  is product of two primes $ pq. $
\item $ \partial^{L}_{2n-1}(\mathcal{P}(D_{2n}))=2n, ~ \partial^{L}_{2n-1-i}(\mathcal{P}(D_{2n}))=3n $ for $ i=1,2,\dots, \phi(n) $, $ \partial^{L}_{2n-1-\phi(n)-i}(\mathcal{P}(D_{2n}))\geq 3n $ for $ i=1,2,\dots,(\phi(n)-2n-2) $ and $ \partial^{L}_{2n}(\mathcal{P}(D_{2n}))=0 $, provided $ n $ is neither prime nor product of two primes $ pq. $
\item $ \partial^{L}_{4n-1}(\mathcal{P}(Q_{n}))=\partial^{L}_{4n-2}(\mathcal{P}(Q_{n}))=4n, ~ \partial^{L}_{4n-2-i}(\mathcal{P}(Q_{n}))=6n $ for $ i=1,2,\dots, 2n-3 $, $ \partial^{L}_{2n+1-i}(\mathcal{P}(Q_{n}))= 8n-4 $ for $ i=1,2,\dots,n) $ $ \partial^{L}_{n+1-i}(\mathcal{P}(Q_{n}))=8n-2 $ for $ 1,2,\dots,n $ and $ \partial^{L}_{4n}(\mathcal{P}(Q_{n}))=0 $, if  $ Q_{n} $ is the generalized quaternion.
\item $ \partial^{L}_{4n-1}(\mathcal{P}(Q_{n}))=\partial^{L}_{4n-2}(\mathcal{P}(Q_{n}))=4n, ~ \partial^{L}_{4n-2-i}(\mathcal{P}(Q_{n}))\geq 6n $ for $ i=1,2,\dots, 4n-3 $, and $ \partial^{L}_{4n}(\mathcal{P}(Q_{n}))=0 $, if  $ n $ is not a power of $ 2 $.
\end{itemize}

{ The \emph{vertex connectivity} $ \kappa(G) $ of a graph $ G $ is the minimum value of $|S|$ among all $S\subseteq V(G)$ such that $ G\setminus S $ is either disconnected or trivial.} The following conjecture about the Laplacian spectral properties of the Laplacian matrix of $ \mathcal{P}(\mathcal{G}) $ is given by Panda \cite{panda}.
\begin{conjecture}\label{Lap conjecture}
For an integer $ n\geq 2 $, the following statements are equivalent:
\begin{itemize}
\item[\bf (i)] $ \mu_{n-1} $ of $ \mathcal{P}(\mathbb{Z}_{n}) $ is an integer.
\item[\bf (ii)] $ \mathcal{P}(\mathbb{Z}_{n}) $ is Laplacian integral.
\item[\bf (iii)] $ n $ is either prime power or product of two distinct primes.
\end{itemize}
\end{conjecture}

Keeping in view Lemma \ref{diameter 2 lemma}, we put forward the following conjecture.
\begin{conjecture}\label{Dist Lap Conjecture}
For an integer $ n\geq 2 $, the following statements are equivalent:
\begin{itemize}
\item[\bf (i)] $ \partial_{1}^{L} $ of $ \mathcal{P}(\mathbb{Z}_{n}) $ is an integer;
\item[\bf (ii)] $ \mathcal{P}(\mathbb{Z}_{n}) $ is distance Laplacian integral;
\item[\bf (iii)] $ n $ is either prime power or a product of two distinct primes.
\end{itemize}
\end{conjecture}

Also by Lemma \ref{diameter 2 lemma} and from Theorems 2.7, 2.9, and 2.12 of Chattopadhyay and Panigrahi \cite{sriparna} and from Lemma 9 of \cite{panda}, we have the following results.
\begin{itemize}
\item $ \partial_{1}^{L}(\mathcal{P}(\mathbb{Z}_{n}))\geq 2n-\phi(n)-p^{\alpha-1}q^{\beta-1} $ with $ p^{\alpha}q^{\beta}, (p<q) $ are primes, equality holds if and only if $ \alpha=\beta=1. $
\item $ \partial_{1}^{L}(\mathcal{P}(\mathbb{Z}_{n}))\geq 2n-\phi(n)-p-q+1 $ with $ n=pqr $ and $ p>q>r $ being primes.
\item $ \partial_{1}^{L}(\mathcal{P}(\mathbb{Z}_{n}))\leq 2n-\phi(n)-1 $, equality holding if and only if $ n $ is prime or product of two distinct primes.
\item $ 8n-2\leq \partial_{1}^{L}(\mathcal{P}(Q_{n}))< 8n-1. $
\end{itemize}

Noting that $ \mu_{n-1}\leq  \kappa(G) $ (see Fiedler \cite{fiedler}) and using Lemma \ref{diameter 2 lemma}, we have
\begin{itemize}
\item From Theorem \ref{theorem zr x Fp,q} and its consequence, $ \kappa(\mathbb{Z}_{r}\times F_{p,q})=r $,  $ \mu_{n-1}(\mathcal{P}(\mathbb{Z}_{r}\times F_{p,q}))\leq r, $ and $ \partial_{1}^{L}(\mathcal{P}(\mathbb{Z}_{r}\times F_{p,q}))\geq 2pqr-r. $
\item From Theorem \ref{theorem F p, qr}, $ \kappa(\mathbb{Z}_{r}\times F_{p,qr})=r $, if $ r=3  $ or $ q=3 $, $ \mu_{n-1}(\mathcal{P}(\mathbb{Z}_{r}\times F_{p,qr}))\leq r, $ and $ \partial_{1}^{L}(\mathcal{P}(\mathbb{Z}_{r}\times F_{p,qr}))\geq 2pqr-r $, if $ r=3 $ or $ q=3. $

\end{itemize}

It still remains open to characterize the extremal graphs with spectral invariants (like energy, spectral radius, algebraic connectivity, energy like invariant, Estrada index)  in  $G(p,q,r)$. This problem can be considered in the future work.

\vskip 3mm
\noindent{\bf Acknowledgements.} The research of Victor A. Bovdi was supported by UAE  grant  G00004159.\\

\noindent\textbf{Data Availability:} There is no data associated with this article.
\vskip 5mm

\noindent\textbf{Conflict of interest} The authors declare that they have no competing interests.


\end{document}